\documentclass[11pt,reqno]{amsart}

\usepackage{amssymb}
\usepackage{amscd}
\usepackage{amsfonts}
\usepackage{mathrsfs}
\usepackage{setspace}
\usepackage{version}
\usepackage[pdftex,colorlinks]{hyperref}

\usepackage{hyperref}

\numberwithin{equation}{section} \theoremstyle{plain}
\newtheorem{theorem}[subsection]{Theorem}
\newtheorem{proposition}[subsection]{Proposition}
\newtheorem{lemma}[subsection]{Lemma}

\newtheorem{definition}[subsection]{Definition}

\renewcommand{\leq}{\leqslant}
\renewcommand{\geq}{\geqslant}
\newsavebox{\proofbox}
\savebox{\proofbox}{\begin{picture}(7,7)  \put(0,0){\framebox(7,7){}}\end{picture}}

\newcommand\Z{\mathbb{Z}}
\newcommand\R{\mathbb{R}}

\newcommand\C{\mathbb{C}}
\newcommand\N{\mathbb{N}}

\newcommand\car{\operatorname{char}}

\newcommand\GL{\operatorname{GL}}

\newcommand\SO{\operatorname{SO}}

\newcommand\F{\mathbb{F}}

\newcommand\U{\mathbb{U}}

\begin{document}

\title[Jordan's original proof]{An exposition of Jordan's original proof of his theorem on finite subgroups of $\GL_n(\C)$.}

\author{
Emmanuel Breuillard}
\address{Emmanuel Breuillard\\
Mathematical Institute\\
Woodstock Road\\
OX2 6GG\\
United Kingdom}
\email{breuillard@maths.ox.ac.uk}

\begin{abstract}
We discuss Jordan's theorem on finite subgroups of invertible matrices and give an account of his original proof.
\end{abstract}

\dedicatory{Dedicated to Udi Hrushovski on the occasion of his 60th birthday.}

\maketitle

\section{introduction}

In 1878 Camille Jordan \cite{jordan} proved the following theorem:

\begin{theorem}[Jordan's theorem] \label{main1}
Let $G$ be a finite subgroup of $\GL_n(\C)$, then there is a normal abelian subgroup $A$ in $G$ of index bounded by a constant $J(n)$ depending on $n$ only.
\end{theorem}

It is the purpose of this note to provide an account of Jordan's original proof of his result. Jordan's proof is purely algebraic, and  quite different from the proofs found in most textbooks (such as \cite{curtis-reiner} or \cite{dixon}) that are based a geometric argument due to Bierberbach \cite{bieberbach}. Jordan's proof does not appear to have been discussed much elsewhere (with the exception of Dieudonn\'e's notes in Jordan's collected works \cite{dieudonne}) even as this year marks the hundredth anniversary of Jordan's death.

Jordan's motivation for proving this result came from the study of linear differential equations of order $n$ with rational functions as coefficients and with algebraic solutions: in this context finite subgroups of $\GL_n$ arise naturally as monodromy groups and information such as Theorem \ref{main1} on the monodromy group translates immediately into structural properties for the solutions of the equation\footnote{For the full story of the motivations and context in which Jordan's theorem was proven, we refer the reader to the wonderful book by Jeremy Gray \cite{gray}}. Prior to Jordan, Fuchs and Klein had studied the two dimensional case and Klein had given a complete list of finite subgroups of $\GL_2(\C)$. Jordan announced his result in \cite{jordan-announce}, published it in \cite{jordan} and later wrote a second article \cite{jordan-napoli} to clarify his proof. 

%he studied linear differential equations of the type $\sum A_k(t)\frac{d^ku}{dt^k} =0$, whose coefficients $A_k(t)$ are rational functions in $\C(t)$ and whose solutions are algebraic (i.e. satisfy a polynomial equation with coefficients in $\C(t)$). He showed that there is a basis of solutions $u_i$ such that for some integers $a_i \in \N$, the $u_i^{a_i}$ belong to $\C[t,(\alpha_j)_j]$, where the $\alpha_j(t)$ are roots of some fixed polynomial equation with coefficients in $\C(t)$ of degree bounded in terms of $n$ only. This result is easily seen to be equivalent to Jordan's theorem as stated above.

Jordan argued by induction on the dimension, but he gave no explicit bound on $J(n)$ in his article, not even an inductive one. 
It is therefore understandable that mathematicians sought to find explicit bounds closer to the truth and this topic has been quite active in the last 146 years. Indeed after Jordan's memoir, several authors gave new proofs of his theorem. The first of these appears to be Blichfeldt, who gave an entirely different proof of Jordan's result via the study of the $p$-Sylow subgroups, for which he established explicit bounds on their size in terms of $p$ and $n$ (see \cite{blichfeldt, miller-blichfeldt-dickson}). Subsequently Bieberbach \cite{bieberbach} came up with yet another very different and  purely geometric argument, which was later refined by Frobenius \cite{frobenius}. This third proof is much slicker and it is the one that can be found, with some variants, in most textbooks that treat the question
(such as \cite[chap. V]{curtis-reiner}, \cite[chap. 8]{raghunathan}). Blichfeldt himself later combined it with his previous approach to improve his bounds on $J(n)$ (see \cite{miller-blichfeldt-dickson} and \cite{speiser,isaacs, dornhoff}).

Bierberbach's argument starts with what people refer nowadays as Weyl's unitary trick (i.e. the observation that a finite, or compact, subgroup $G$ of $\GL_n(\C)$ can be conjugated inside the compact unitary group $\U_n(\C)$ by averaging a hermitian product over $G$). Then one makes use of a volume packing argument in combination with the \emph{commutator shrinking property} of Lie groups, i.e. the fact that commutators of elements close to the identity in $\U_n(\C)$ are themselves close, and in fact much closer, to the identity. This commutator shrinking property has inspired several other authors (\cite{zassenhauss}, \cite{boothby-wang}, \cite{kazhdan-margulis}, \cite{bass}) and is nowadays a crucial tool in the study of discrete subgroups of Lie groups and in Riemannian geometry. We refer the reader to \cite[Thm A]{robinson} or \cite[\S 2]{breuillard-green} for a proof of Jordan's theorem via this argument.

Jordan's original proof on the other hand was based on a purely algebraic idea, that should be traced back to Klein's method for the classification of the finite subgroups of rotations of the $2$-sphere (and isometries of Plato's solids), as described in Klein's famous book on the icosahedron  \cite{klein}. Basically, one enumerates the elements of $G$ according to the shape and size of their centralizers and one can thus write a class equation involving the order of $G$ and of the centralizers of its elements. Inducting on dimension this yields a diophantine equation of the form:
\begin{equation}
  \frac{1}{g}= \frac{1}{q_1}+...+\frac{1}{q_{k}} -\frac{b}{a},
\end{equation}
where $g=[G:\Phi]$ is the index of the center $\Phi$ of $G$ in $G$, $a,b$ and $k$ are integers that are bounded in terms of $n$ only and each $q_i$ is the cardinality of a certain subgroup of $G/\Phi$. It is easy to check that any equation of this form forces $g$ to be bounded (in terms of $n$) and Jordan then discusses the boundedly many cases that may arise. Although more cumbersome, this method gives potentially much more information on the finite subgroup $G$. For example, Jordan used it to list all finite subgroups of $\GL_3(\C)$, giving an explicit set of generators for each class of groups, after examining some $47$ different cases\footnote{In fact Jordan missed some groups, see \cite{blichfeldt, dieudonne}.}.

%In the last twenty or so years, a lot of attention has been devoted to Jordan's theorem by various authors, in particular to the question of its extension to characteristic $p$ (\cite{brauer-feit}, \cite{weisfeiler}). 
With the advent of the classification of finite simple groups, B. Weisfeiler \cite{weisfeiler} and more recently M. Collins \cite{collins} have found tight bounds for $J(n)$. For example, Collins proved that if $n \geq 71$, then $J(n)$ can be taken to be $(n+1)!$. This is tight, i.e. $(n+1)!$ is always a lower bound for $J(n)$, because the symmetric group on $n+1$ letters acts irreducibly on the hyperplane $\sum_{i=1}^{n+1} x_i=0$ by permuting the $n+1$ coordinates.

Schur \cite{schur1911} extended Jordan's theorem, proving that it holds assuming only that the group $G$ is torsion (i.e. every element has finite order). In particular every finitely generated torsion subgroup  of $\GL_n(\C)$ is finite. This is sometimes called the Jordan-Schur theorem, see \cite{wehrfritz, curtis-reiner}.

In another direction initiated by Brauer and Feit \cite{brauer-feit},  Larsen and Pink \cite{larsen-pink} gave a vast generalization of Jordan's theorem to finite linear groups in characteristic $p$, which avoids the classification of finite simple groups. Interestingly enough, part of their proof is very much akin to Jordan's original argument. See also \cite{borovik} for a recent use of theorems of Jordan and Larsen-Pink type in the study of finite group actions on elementary abelian $p$-groups with finite Morley rank. 
%We will give a short exposition of their result after we discuss Jordan's theorem.

Finally, we mention that there are non-linear analogues of Jordan's theorem for finite subgroups of homeomorphisms of manifolds (conjectured by E. Ghys) and for finite subgroups of birational automorphisms of algebraic varieties. In these cases, it has recently been shown that there is a nilpotent subgroup of index bounded only in terms of the dimension of the manifold or variety. But ``nilpotent" cannot be replaced by ``abelian''.  We refer the reader to the preprints \cite{csikos-pyber-szabo} and \cite{guld} and references therein for these exciting recent developments. 

\smallskip

This paper is organized as follows: in Section \ref{sec2} we define the notion of $M$-fan and state the version of Theorem \ref{main1} that will be used as induction hypothesis. In Sections \ref{sec3} and \ref{sec4} we write the corresponding class equation and complete Jordan's proof of Theorem \ref{main1}. In Section \ref{sec5} we give an illustration of the method by specializing to the case of $n=2$ and we derive the classical results of Klein for finite subgroups of $\SO_3(\R)$. In Section \ref{sec6} we discuss a non-standard treatment of Jordan's proof, which is very close to Jordan's original formulation of his proof, and in the last section we briefly survey bounds for $J(n)$ from a historical perspective.

\vspace{.5cm}
\noindent \emph{Acknowledgements.} I am grateful to Udi Hrushovski, who first drew my attention to Jordan's original memoir, to Michael Collins and Terry Tao for useful comments and to Jean-Pierre Serre for encouraging me to publish this note, a first version of which was written some ten years ago. I also thank the referee for their thoughtful comments and careful reading.

\bigskip

\section{A reformulation of Theorem \ref{main1}}\label{sec2}

As we will see below, Jordan's argument uses nothing about the field $\C$ and in fact his proof carries over to an arbitrary field provided we assume that every element of $G$ is semisimple, i.e., diagonalizable in some field extension. So we let $K$ be an arbitrary field, which we assume algebraically closed without loss of generality.

Let us first reformulate Theorem \ref{main1} in the form orginally proved by Jordan. For this we need to introduce a couple of definitions.

%dictionnary:

%A est echangeable a B = A commutes with B

%A est permutable a G = A normalizes G

%s est afferente au faisceau F = F is the maximal $M$-fan of the centralizer Z(s).

\begin{definition} By a  \emph{root torus}, we mean a subgroup of $\GL_n(K)$ which is conjugate to a subgroup of the diagonal matrices defined by a set of equalities between the diagonal entries.
\end{definition}

For example, the subgroup of diagonal matrices $\{g=diag(a_1,...,a_6) | a_i\in K^*, a_1=a_2, a_5=a_6\}$ is a root torus of $\GL_6(K)$.

\begin{definition} Let $G$ be a finite subgroup of $\GL_n(K)$. Given $M \geq 2$, we say that a subgroup $F$ of $G$ is an \emph{$M$-fan} if it is conjugate to a subgroup of the diagonal matrices $\{g=diag(a_1,...,a_n), a_i \in K^*\}$ such that for every pair of indices $i,j$ the set of ratios $a_i(g)/a_j(g)$ is either reduced to $\{1\}$ or achieves at least $M$ distinct values as $g$ varies in $F$.
\end{definition}

The terminology \emph{fan} is a liberal translation of Jordan's \emph{faisceau}\footnote{We are grateful to the referee for suggesting this translation. In fact the word \emph{faisceau} is used throughout Jordan's other works to mean sometimes `subgroup', sometimes `abelian subgroup'.}. Note that the subgroup $\Phi$ of all scalar matrices in $G$ is clearly an $M$-fan, for any $M\ge 2$.

Note that every $M$-fan $F$ is contained in a unique minimal root torus $S_F$ defined by the same equalities between diagonal elements, such as $a_i=a_j$, as those that hold in $F$. In particular $G\cap S_F$ is itself an $M$-fan and every maximal $M$-fan in $G$ has this form.

\smallskip

We can now formulate an alternative, slightly more precise, version of Theorem \ref{main1} above:

\begin{theorem}[Jordan's theorem, second form]\label{main2} Given $n \in \N$, there are constants $M=M(n),N=N(n) \geq 1$ such that the following holds.
Let $K$ be an algebraically closed field, and let $G$ be a finite subgroup of $\GL_n(K)$, such that every element of $G$ is diagonalizable. Then $G$ contains a unique maximal $M$-fan. Call it $\mathcal{F}$. We have $[G:\mathcal{F}] \leq N$.
%Then there is a root torus $S$ in $\GL_n(K)$ such that $G$ belongs to the normalizer $N(S)$ of $S$ and $S \cap G$ has index at most $N_n$ in $G$, where $N_n$ is a function of $n$ only. Moreover if $M$ is greater than a bound depending on $n$ only, then every $M$-fan in $G$ is contained in $S$.
\end{theorem}

The proof of Theorem \ref{main2} will span the next two sections. Before we start, a number of simple remarks are in order:

\smallskip

1) Since $\mathcal{F}$ is unique, it must be normal in $G$.%, so $\mathcal{F}$ is a normal abelian subgroup of $G$.

\smallskip

2) To see that Theorem \ref{main2} implies Theorem \ref{main1}, it only remains to check that if $K=\C$, then every element of $G$ is diagonalizable: this is indeed the case, because every element of $G$ has finite order and is thus diagonalisable over $\C$.

%From the theorem, it follows that there is a unique maximal $M$-fan in $G$. Indeed, the subgroup of $G$ generated by its $M$-fans lies in $S$ and is itself an $M$-fan. We denote it by $\mathcal{F}_G$. We claim that $[G:\mathcal{F}_G] \leq N_nM^{n-1}$. To see this, write $S$ as a diagonal subgroup of $\GL_n(K)$, with elements $s=\text{diag}(s_1,...,s_n)$. If for two indices $i$ and $j$ one has $s_i=s_j$ for all $s \in \mathcal{F}_G$, then from the definition of $\mathcal{F}_G$ the quantity $s_i/s_j$ can achieve at most $M$ values when $s$ varies in $S \cap G$. Since $S$ is defined by at most $n-1$ root equations of the type $s_i=s_j$, we must have $[S \cap G : \mathcal{F}_G] \leq M^{n-1}$, and this proves the claim.

\smallskip

{3) Although we will prove the result in any characteristic, it is worth mentioning that the case of positive characteristic follows from the case when $K=\C$, because if $G$ is as in Theroem \ref{main2}, then $|G|$ is prime to $\car(K)$ and thus $G$ admits an embedding in $\GL_d(\C)$. See for instance \cite[Proof of Thm C]{nori}, or \cite[Theorem 3.8]{dixon}.

\smallskip

4) The proof of Theorem \ref{main2} will proceed by induction on the dimension. The letter $F$ will denote a fan and we will reserve the letter $\mathcal{F}$ for maximal fans.

\smallskip

5) Since $S_\mathcal{F}$ is normalized by $G$, $G$ must permute the eigenspaces of $S_\mathcal{F}$. So if $G$ acts primitively on $K^n$ (i.e. does not permute the components of any non-trivial direct sum decomposition of $K^n$), then $S_\mathcal{F}$ must be reduced to scalar matrices and those have bounded index in $G$.

\smallskip

6) If $g \in GL_n(K)$ normalizes $\mathcal{F}$, then it must normalize the root torus $S_\mathcal{F}$ too. In particular $G$
lies in the normalizer of a root torus $S_\mathcal{F}$ and $[G:G \cap S_\mathcal{F}] \leq N$.
\smallskip

\vspace{.5cm}

As seen from items 1) and  2) above, Theorem \ref{main2} implies Theorem \ref{main1}. It turns out that one can also derive Theorem \ref{main2} from Theorem \ref{main1} directly and we explain this in the paragraph below. To be more precise, since we have only stated Theorem \ref{main1} over $\C$ while Theorem \ref{main2} is also valid in positive characteristic,  we are going to prove that Theorem \ref{main2} follows from the assertion that any finite subgroup of $\GL_n(K)$ made of diagonalizable elements admits a normal abelian subgroup of index at most $J(n)$.  Jordan's original proof goes by proving Theorem \ref{main2} first, because its formulation is more adequate for the induction scheme.

\smallskip

\noindent \emph{Proof of the equivalence of Theorems \ref{main1} and \ref{main2}}. Assume the conclusion of Theorem \ref{main1}. Since every element of $G$ is diagonalizable and $A$ is abelian, $A$ is simultaneously diagonalizable and $K^n$ decomposes as a direct sum of weight spaces (i.e. joint eigenspaces) for $A$. Since $A$ is normal in $G$, these eigenspaces are permuted by $G$ and thus $G$ lies in the normalizer $N(S)$ of the root torus  $S$ that acts on $K^n$ by a scalar multiple on each one of the weight spaces of $A$. Note that $A \leq S$. Moreover, if $F$ is an $M$-fan with $M>J(n)$ and $m:=[F:F \cap S] \leq [G:S] \leq J(n)$, we have $f^m \in S$ for all $f \in F$. Thus $F \cap S$ is an $M/m$-fan lying in $S$. Since $M/m > 1$, this implies that $F$ itself lies in $S$. Hence every $M$-fan is contained in $S$. Finally, viewing $S$ as a diagonal subgroup it is straightforward to check that the subgroup generated by all $M$-fans in $S$ is itself an $M$-fan.  Hence $G \cap S$ is the unique maximal $M$-fan in $G$. This completes the claims of Theorem \ref{main2} with $\mathcal{F}:=G \cap S$, $N=J(n)$, $M=N+1$.

\section{Jordan's fundamental equation}\label{sec3}

In this section we begin the proof of Theorem \ref{main2} and obtain Jordan's fundamental equation $(\ref{fund})$ below, which expresses an enumeration of the elements of $G$ into various classes which we are about to describe. The proof of Theorem \ref{main2} will be completed in the next section after a discussion of the fundamental equation.

\smallskip
We will proceed by induction on the dimension $n$.

If $n=1$, then $\GL_1(K)=K^{*}$ is abelian and there is nothing to prove. We now assume the theorem proven for all dimensions $< n$.

Observe that, by the argument at the end of the last section, it will be enough to establish the conclusion of Theorem \ref{main1}, namely the existence of an abelian normal subgroup of  index bounded by some function $J(n)$, as this will automatically imply the conclusion of Theorem \ref{main2} with $N(n)=J(n)$ and $M(n)=N(n)+1$.

If $G$ preserves a direct sum decomposition $K^n=K^{r} \oplus K^{n-r}$, with $1 < r < n$, then we may use the induction hypothesis in the obvious way applying it to the projections $\pi_r(G)$ and $\pi_{n-r}(G)$ to $\GL_r(K)$ and $\GL_{n-r}(K)$ respectively. The conclusion of Theorem \ref{main1} then easily follows as soon as $J(n) \geq \max_{0<r<n}{N(r)N(n-r)}$ and that of Theorem \ref{main2} too as we have just said.

We will use repeatedly the last observation for subgroups of $G$ that preserve such a decomposition. If $g\in G$ is not a scalar matrix, then the centralizer $C_G(g)$ preserves the eigenspace decomposition of $g$ on $K^n$. We can therefore apply this observation to $C_G(g)$ and conclude from the induction hypothesis that $C_G(g)$ contains a unique maximal $M$-fan (for all $M$ larger than a number depending on $n$ only). That is:

\begin{lemma}\label{max} If $g \in G$ is not a scalar matrix, then the centralizer $C_G(g)$ contains a unique maximal $M$-fan.
\end{lemma}

We can thus set the following definition:

\begin{definition} An element $g$ is said to be \emph{associated with} an $M$-fan $F$ if $F$ lies in the centralizer $C_G(g)$ and is the unique maximal $M$-fan of $C_G(g)$.
\end{definition}

We denote by $F_g$ the $M$-fan associated with $g$. This definition makes sense (so far, thanks to the induction hypothesis) as soon as $g$ is not a scalar matrix in $\GL_n(K)$ by the remarks above the definition. Note that, by maximality, $F_g$ must contain the subgroup $\Phi$ of $G$ of all scalar matrices in $G$. Moreover, setting $N:=N(n-1)$, it follows from the induction hypothesis that

\begin{equation}\label{borne}
[C_G(g) : F_g]\leq N.
\end{equation}

\vspace{.5cm}
These remarks also have the following three consequences:

\begin{lemma}\label{max0}If $F$ is an $M$-fan of $G$ not entirely made of scalar matrices, then $F$ is contained in a unique maximal $M$-fan $\mathcal{F}$ of $G$.
\end{lemma}
\begin{proof} Let $f \in F$ be a non-scalar element. If $F_1$ is an $M$-fan containing $F$, then $F_1$ must commute with all elements of $F$ and thus lie in $C_G(f)$, the centralizer of $f$. Therefore $F_1$ must lie in the unique maximal $M$-fan of $C_G(f)$. \end{proof}

Let $F$ be an $M$-fan of $G$ not contained in the scalar matrices $\Phi$ and let $\mathcal{F}$ be the maximal $M$-fan of $G$ containing $F$. Since $\mathcal{F}$ is contained in the centralizer $C_G(F)$, it must be the maximal $M$-fan there too and, by the induction hypothesis, we must have $[C_G(F):\mathcal{F}] \leq N$.

\begin{lemma}\label{enef} Suppose $\Phi \lneqq F \lneqq \mathcal{F}$. Then the number $n_F$ of elements of $G$ associated with $F$ is divisible by $|\mathcal{F}|$ and $\frac{n_F}{|\mathcal{F}|} \leq N$.
\end{lemma}

\begin{proof} If $n_F=0$ there is nothing to prove, so we assume $n_F \geq 1$. Every element associated with $F$ lies in the centralizer $C_G(F)$. Moreover, if $g \in C_G(F)$ is associated with $F$ and $f \in \mathcal{F}$, then $gf$ is also associated with $F$, i.e. $F_{gf}=F_g=F$. Indeed, since $F \lneqq \mathcal{F}$ we must have $gf \notin \Phi$ (as otherwise $C_G(g)=C_G(f)$ contains $\mathcal{F}$) and by Lemma \ref{max} there is a unique maximal $M$-fan $F_{gf}$ in $C_G(gf)$. Since $F \subset C_G(gf)$ we have $F \subset F_{gf} \subset C_G(gf)$. Moreover $F_{gf}$ is contained in $\mathcal{F}$ and must therefore commute with $f$, and hence also with $g$. It follows that $F \subset F_{gf} \subset C_G(g)$ and $F=F_{gf}$ by maximality of $F$.

Consequently, the set of elements of $G$ associated with $F$ is a union of cosets of $\mathcal{F}$ all lying in $C_G(F)$. Since $C_G(F)$ contains $\mathcal{F}$ as a subgroup of index at most $N$ the result follows.
\end{proof}

And for maximal fans we have:

\begin{lemma}\label{maxfan} Let $\mathcal{F}\neq \Phi$ be a maximal $M$-fan in $G$. Then the number $n_\mathcal{F}$ of non-scalar elements $g$ in $G$ which are associated with $\mathcal{F}$ is $n_\mathcal{F}=|C_G(\mathcal{F})| - |\Phi|$, and $[C_G(\mathcal{F}):\mathcal{F}] \leq N$.
\end{lemma}

\begin{proof} A non-scalar element $g$ is associated with $\mathcal{F}$ if and only if $\mathcal{F} \leq C_G(g)$, i.e. $g \in C_G(\mathcal{F})$. %Indeed the union of $\Phi$ and the subset of non-scalar elements of $G$ that are associated with $\mathcal{F}$ is a union of cosets of $\mathcal{F}$ inside $C_G(\mathcal{F})$ for the same reason as in the proof of Lemma \ref{enef}. Since $\frac{[C_G(\mathcal{F}):\mathcal{F}]}{|\mathcal{F}|} \leq N$,
The bound follows from $(\ref{borne})$.
\end{proof}

%But conversely, every element $g \in C_G(F)$ is associated to some $M$-fan $F_g$ with $F \subset F_g \subset \mathcal{F}$. Indeed, if $g \in C_G(F) \setminus \Phi$, then $F$ lies in the centralizer $C_G(g)$ and is thus contained in the unique maximal $M$-fan $F_g$ of $C_G(g)$. By uniqueness of $\mathcal{F}$, we conclude that $F_g \subset \mathcal{F}$. If $g$ is scalar, then $g \in F$ since $\Phi \subset F$.

%\end{proof}

\vspace{.5cm}

The strategy of Jordan's proof consists in enumerating the elements of $G$ according to their associated $M$-fan. Let $\Phi$ be the scalar matrices in $G$. We may decompose $G$ as the \emph{disjoint} union $$G=\Phi \cup_{F} \{g\text{, } g \text{ associated with }F \}$$ where the union is taken over fans arising as maximal $M$-fans of centralizers of non scalar elements of $G$. We will split this union into four disjoint parts, $$G=\Phi \cup G_1 \cup G_2 \cup G_3,$$ where $G_1$ is the subset of $g$'s not in $\Phi$ such that $F_g=\Phi$, $G_2$ is the subset of $g$'s not in $\Phi$ such that $F_g$ contains $\Phi$ strictly but is not the maximal $M$-fan $\mathcal{F}_g$ which contains it by Lemma \ref{max0}, and finally $G_3$ is the remaining subset of those $g$'s not in $\Phi$ for which $F_g$ is not $\Phi$ and is maximal in $G$. We now consider each subset $G_i$ one after the other.
\bigskip

1) We first enumerate the elements of $G_1$, that is the $g$'s outside $\Phi$ which are associated with $\Phi$. This subset is invariant under conjugation by $G$. Also it is clearly a union of cosets of $\Phi$, for if $\phi \in \Phi$, then $C_G(g\phi)=C_G(g)$ and thus $g\phi$ is also associated with $\Phi$. It follows that conjugation by $G$ permutes those $\Phi$-cosets.

The stabiliser $N_G(g\Phi)$ of a $\Phi$-coset $g\Phi$ under the $G$-action by conjugation must have index at most $n$ in the $C_G(g)$. Indeed,  if $h\in G$ has $hg\Phi h^{-1}=g\Phi$, then $hgh^{-1}=g\phi$ for some $\phi \in \Phi$. It follows that $\det(\phi)=1$ and thus $\phi$ is an $n$-th root of unity. We conclude that $[N_G(g\Phi):C_G(g)] \leq n$.

Therefore the number of elements in the $G$-conjugacy class of the coset $g\Phi$ equals $$\frac{|G|}{N_G(g\Phi)} |\Phi|= |G| \frac{1}{[N_G(g\Phi):C_G(g)]} \frac{1}{[C_G(g) : \Phi]}=|G| \frac{1}{\lambda}.$$ Enumerating all such conjugacy classes, we find: $$|G_1|=|G|\left( \frac{1}{\lambda_1} + ... + \frac{1}{\lambda_{k_1}}\right).$$
where each $\lambda_i$ is a positive integer of size at most $nN$ by $(\ref{borne})$ and the remark above.

\bigskip

2) We now pass to the subset $G_2$. Clearly $G_2$ is stable under conjugation by $G$. Let $F$ be an $M$-fan of $G$ with maximal $M$-fan $\mathcal{F}$ such that $\Phi \lneqq F \lneqq \mathcal{F}$. Let $n_F$ be the number of $g$'s which are associated with $F$. By Lemma \ref{enef}, $n_F/|\mathcal{F}|$ is an integer of size at most $ N$.

Grouping together the fans that are conjugate to $F$, we obtain $\frac{|G|}{|N_G(F)|}$ different fans, where $N_G(F)$ is the normalizer of $F$ in $G$. Note that \begin{equation}\label{normali}[N_G(F):C_G(F)] \leq n!\end{equation} since $N_G(F)$ permutes the weight spaces of $F$ and hence a subgroup of index at most $n!$ will preserve them and thus commute with $F$.

It follows that the number of elements that are associated with a fan lying in the $G$-conjugacy class of $F$ equals $$ n_F \frac{|G|}{|N_G(F)|}=|G| \frac{1}{[N_G(F) : C_G(F)]} \frac{ n_F/|\mathcal{F}|}{[C_G(F) : \mathcal{F}]}=|G| \frac{\nu}{\mu},$$ and thus enumerating the different conjugacy classes $$|G_2|=|G|\left( \frac{\nu_1}{\mu_1} + ... + \frac{\nu_{k_2}}{\mu_{k_2}}\right), $$ where the $\nu_i \leq N$ and  $\mu_i\leq n!N$  are positive integers.

\bigskip

3) Finally we consider the subset $G_3$ of those non-scalar $g$'s such that $F_g$ is maximal in $G$ and different from $\Phi$. Clearly this set is invariant under conjugation by $G$. Given a maximal $M$-fan $\mathcal{F}$, the number $n_\mathcal{F}$ of elements of $G$ which are associated with $\mathcal{F}$ equals $|C_G(\mathcal{F})| - |\Phi|$ according to Lemma \ref{maxfan}.

Setting $\omega=[N_G(\mathcal{F}):C_G(\mathcal{F})]$ and $q=[N_G(\mathcal{F}):\Phi]$, the number of elements that are associated with a maximal fan conjugate to $\mathcal{F}$ is $$ n_\mathcal{F} \frac{|G|}{|N_G(\mathcal{F})|}= |G|\left(\frac{1}{\omega}-\frac{1}{q}\right),$$ where $\omega$ and $q$ are positive integers with $\omega \leq n!$ and $q=[C_G(\mathcal{F}):\Phi]\omega  \ge 2\omega$.

%Setting $q=\frac{|\mathcal{F}|}{|\Phi|}$, we get $n_\mathcal{F}=|\mathcal{F}|(\rho - \frac{1}{q})$. Therefore the number of elements that are associated with a maximal fan conjugate to $\mathcal{F}$ is $$ n_\mathcal{F} \frac{|G|}{|N_G(\mathcal{F})|}= |G|\frac{1}{[N_G(\mathcal{F}):\mathcal{F}]} \left(\rho- \frac{1}{q}\right)=|G|\frac{1}{\eta}\left(\rho- \frac{1}{q}\right),$$ where $\eta$ is a positive integer of size at most $n!N$, $\rho$ is a positive integer of size at most $N$.

Summing over the conjugacy classes, we get: $$ |G_3|= |G| \left( \left(\frac{1}{\omega_1}- \frac{1}{q_1}\right)+ ... + \left(\frac{1}{\omega_{k_3}}- \frac{1}{q_{k_3}}\right) \right).$$

\bigskip

Combining all three cases we have thus completed our enumeration of $G$ and we obtain:

\begin{proposition}[Jordan's fundamental equation]\label{jor} Let $G$ be a finite subgroup of $\GL_n(K)$ all of whose elements are diagonalisable, and $\Phi$ the subgroup of scalar matrices in $G$. Then there are positive integers $q_i$ dividing $g:=|G|/|\Phi|$ such that

%\begin{equation}\label{fund}
%|G|= |\Phi| + |G|\left( \frac{1}{\lambda_1} + ... + \frac{1}{\lambda_{k_1}}\right) + |G|\left( \frac{\nu_1}{\mu_1} + ... + \frac{\nu_{k_2}}{\mu_{k_2}}\right) +\\
%|G| \left( \frac{1}{\eta_1}\left(\rho_1- \frac{1}{q_1}\right)+ ... + \frac{1}{\eta_{k_3}}\left(\rho_{k_3}- \frac{1}{q_{k_3}}\right) \right).
%\end{equation}

\begin{equation}\label{fund}
|G|= |\Phi| + |G|\sum_{i=1}^{k_1} \frac{1}{\lambda_{i}} +  |G|\sum_{i=1}^{k_2} \frac{\nu_i}{\mu_i} +
|G|\sum_{i=1}^{k_3}
\left(\frac{1}{\omega_i} - \frac{1}{q_i}\right),
\end{equation}
where $k_i$, $\lambda_i$, $\nu_i$, $\mu_i$ and $\omega_i$ are non-negative integers of size at most $2n!N$ (recall that $N=N(n-1)$ is the bound from Theorem \ref{main2} under the induction hypothesis). In particular,% setting $g:=\frac{|G|}{|\Phi|}$ and recalling that $q_i=[N_{G}(\mathcal{F}_i):\Phi]$, we obtain:
\begin{equation}\label{nn}
  \frac{1}{g}= \frac{1}{q_1}+...+\frac{1}{q_{k_3}} -\frac{b}{a},
\end{equation}
where $\frac{b}{a}$ is an irreducible fraction whose numerator and denominator are bounded in terms of $n$ only.
\end{proposition}

To prove Proposition \ref{jor} it remains only to show the bound on the number $k_i$ of elements in each sum and then derive $(\ref{nn})$. But this follows from the equation $(\ref{fund})$ and from the bounds previously obtained, because $\frac{1}{\omega_i} -  \frac{1}{q_i} \geq \frac{1}{2\omega_i}$ for each $i=1,...,k_3$ and thus each term in the above sums contributes at least $|G|/{2n!N}$ forcing $k_1+k_2+k_3 \leq 2n!N$.

Showing $(\ref{nn})$ is a simple matter of rearranging $(\ref{fund})$:
\begin{equation}\label{eq}
\frac{|G|}{|\Phi|} \left(\sum_{i=1}^{k_1} \frac{1}{\lambda_{i}} +  \sum_{i=1}^{k_2} \frac{\nu_i}{\mu_i} +
\sum_{i=1}^{k_3} \frac{1}{\omega_i}   -  1 \right) = \sum_{i=1}^{k_3} \frac{|G|}{|\Phi|}\frac{1}{q_i} - 1 .
 \end{equation}

Then we let $g:=\frac{|G|}{|\Phi|}$ and $\frac{b}{a}:=\sum_{i=1}^{k_1} \frac{1}{\lambda_{i}} +  \sum_{i=1}^{k_2} \frac{\nu_i}{\mu_i} +
\sum_{i=1}^{k_3}
 \frac{1}{\omega_i}   -  1$, where $\frac{b}{a}$ is an irreducible fraction.  We thus get $(\ref{nn})$.

Note further that $a$ is bounded in terms of $n$ only, indeed it cannot exceed the least common multiple of at most $2n!N$ integers of size at most $n!N$. A similar bound holds for $b$. This completes the proof of Proposition \ref{jor}.

\section{Proof of Theorem \ref{main2}}\label{sec4}

It remains to discuss the fundamental equation $(\ref{fund})$ according to the possible values of the integers $\lambda_i$, $\nu_i$, $\mu_i$, $\omega_i$ and $q_i$.

The proof will rest on the following elementary lemma about fractions:

\begin{lemma}\label{fraction} Consider the following equation, where all variables are positive integers:  $$ \frac{1}{g}= \frac{1}{q_1}+...+\frac{1}{q_k} -\frac{b}{a}. $$ Suppose that $q_i < g$ for all $i$, then $ g \leq f(k,a),$ where $f(k,a)$ is a function of $k$ and $a$ only. One may take $f(k,a)=(k!a)^{2^k}$.
\end{lemma}

\begin{proof} The proof proceeds by induction on $k$. If $k=1$, then $\frac{1}{q_1} \geq \frac{b}{a}$ implies $q_1 \leq a$ and $\frac{1}{g} \geq \frac{1}{q_1a} \geq \frac{1}{a^2}$, so $g \leq a^2=:f(1,a)$.

Suppose the lemma proven for all indices $\leq k-1$. Without loss of generality, we may assume that $\frac{1}{q_1} \leq ... \leq \frac{1}{q_k}$. Then $\frac{1}{q_2} + ... + \frac{1}{q_k} < \frac{b}{a}$, for $q_1<g$.% and thus $q_1=g$ forcing $b=0$, which is a contradiction to our standing assumptions.

It follows that $\frac{b}{a} > \frac{1}{q_2} + ... + \frac{1}{q_k} \geq \frac{1}{q_k}$. We may thus write $\frac{c}{d}=\frac{b}{a}-\frac{1}{q_k}$, where $c,d$ are positive integers and $\frac{c}{d}$ is an irreducible fraction.

Since $\frac{1}{g} \leq \frac{k}{q_k} - \frac{b}{a}$, we get $q_k \leq ka$ and thus $d=\text{lcm}(a,q_k) \leq ka^2$. We obtain $$\frac{1}{g}= \frac{1}{q_1}+...+\frac{1}{q_{k-1}} -\frac{c}{d}. $$ Applying the induction hypothesis we conclude that $g \leq f(k-1,ka^2)=:f(k,a)$.

\end{proof}

We now complete the proof of Theorem \ref{main2}. If $k_3=0$, then we see from $(\ref{nn})$ that $\frac{-b}{a}=\frac{1}{g}$ so $b=-1$ and $g=a$ is bounded in terms of $n$ only by Proposition \ref{jor}.  Hence $\Phi$ has bounded index in $G$ and we are done.

%If $M$ is larger than this bound, then there are no $M$-fans except the trivial ones contained in $\Phi$ and Theorem \ref{main2} holds with $\mathcal{F}=\Phi$.

Assume $k_3 \ge 1$. If $q_i=g$ for some $i$, then $G=N_G(\mathcal{F}_i)$ and $[G:\mathcal{F}_i] = \omega_i [ C_G(\mathcal{F}_i),\mathcal{F}_i] \leq n!N$ by $(\ref{borne})$. So $\mathcal{F}_i$ is the desired abelian normal subgroup of bounded index and we are done.
%  in fact there is only one conjugacy class of maximal $M$-fans, i.e. that of $\mathcal{F}_i$. In particular $k_3=1$ and $G$ normalizes the root torus $S_{\mathcal{F}_1}$ associated to $\mathcal{F}_1$ (i.e., satisfying the same equalities between diagonal coefficients as $\mathcal{F}_1$). Moreover we know from $(\ref{borne})$ that $[G: \mathcal{F}_1]$ is bounded in terms of $n$ only. 

The right hand side of $(\ref{eq})$ is a non-negative positive integer. If it is zero, then $k_3=1$, $q_1=g$ and we fall back in the previous case. Otherwise it is positive and thus $b>0$, so that we are in the situation of Lemma \ref{fraction}. We conclude that $g$ is bounded in terms of $n$ only and again we are done.%above by a bound depending only on $k_3$ and $a$, hence on $n$ only. In both cases, if $M$ is larger than this bound, then there are no $M$-fans except the trivial ones contained in $\Phi$ and Theorem \ref{main2} holds with $\mathcal{F}=\Phi$.

%Finally if the right hand side of  $(\ref{eq})$ is $0$, then $k_3=1$, $q_1=g$ and $G=N_G(\mathcal{F}_1)$. This means there is only one conjugacy class of maximal $M$-fans, i.e. that of $\mathcal{F}_1$, and in fact only one maximal $M$-fan since $G$ normalizes $\mathcal{F}_1$. Clearly $G$ normalizes the root torus $S_{\mathcal{F}_1}$ associated to $\mathcal{F}_1$ (i.e., satisfying the same equalities between diagonal coefficients as $\mathcal{F}_1$). Moreover we know from $(\ref{borne})$ that $[G: \mathcal{F}_1]$ is bounded in terms of $n$ only. 

Theorem \ref{main2} is now proven in full.

\bigskip

\begin{remark} We mention in passing that the proof of Landau's theorem \cite{landau} that there are only finitely many finite groups $G$ with exactly $k$ conjugacy classes $\emph{cl}_1,\ldots,\emph{cl}_k$, is based on a similar, and easier, diophantine equation, namely $$1 = \frac{1}{q_1} + \ldots +\frac{1}{q_k}$$ where $q_i=|G|/|\emph{cl}_i|$. This is an instance of an \emph{Egyptian fraction} \cite{bloom-elsholtz}, and a simple argument \cite{newman} implies that $k \ge \frac{1}{\log 4}\log \log |G|$.
\end{remark}

\section{Platonic solids and the finite subgroups of $\SO_3(\R)$}\label{sec5}

As an illustration and for the sake of comparison, we recall in this section a proof of the classification of finite subgroups of $\SO_3(\R)$ following Klein's method, as given in many textbooks, e.g. see \cite{senechal, zassenhauss2}.

Let $G$ be a finite subgroup of $\SO_3(\R)$. Every non trivial element of $G$ is a rotation around some axis. Let $X$ be the set of all axes that arise as axes of rotations in $G$. Clearly $G$ permutes $X$ because if $x_h$ is the axis of $h \in G$, then $gx_h=x_{ghg^{-1}}$. The determination of all possible groups $G$ proceeds via a double counting argument, or class equation, which enumerates the elements of $G$ according to their fixed axis. Given an axis $x \in X$, let $G_x$ be the subset of elements of $G$ whose axis is $x$, to which we adjoin the identity. Then $G_x$ is a subgroup. It usually coincides with the centralizer of $x$, except when $x$ is a flip (i.e. has angle $\pi$). Enumerating the elements of $G$ starting with the identity element, we can write: $$|G|=1+\sum_{x\in X} (|G_x|-1)$$

To go further, we group together the terms corresponding to two axes that are $G$-congruent (i.e. $x \sim y$ if there is $g \in G$ with $y=gx$). We obtain
\begin{equation}\label{eq2}
|G|=1+ \sum_{\text{classes of x }\in X} \frac{|G|}{|Stab_x|}(|G_x|-1),
\end{equation} where $Stab_x$ is the stabilizer of $x$ in $G$. It is a subgroup of $G$. Now observe that an element $g$ of $G$ which preserves $x$ may be only of two possible forms: either it fixes both poles of $x$, in which case $g$ belongs to $G_x$, or it permutes the two poles of $x$. It follows that $G_x$ is a subgroup of $Stab_x$ of index either $1$ or $2$. Let $x_1,...,x_r,x_{r+1},...,x_{r+s}$ be a set of representatives of the $G$-orbits in $X$ such that $[Stab_{x_i}:G_{x_i}]=1$ if $1 \le i \le r$ and
 $[Stab_{x_i}:G_{x_i}]=2$ if $r+1 \le i \le r+s$. Setting $g_i=|G_{x_i}|$, dividing by $|G|$ in $(\ref{eq2})$ we obtain,

\begin{equation}\label{eq3bis}
1=\frac{1}{|G|}+ \sum_{i=1}^r (1-\frac{1}{g_i}) + \frac{1}{2}\sum_{i=r+1}^{r+s} (1-\frac{1}{g_i})
\end{equation}
or equivalently in the form of Jordan's fundamental equation $(\ref{nn})$,
\begin{equation}\label{eq3}
\frac{1}{|G|}=\frac{1}{n_1} +...+ \frac{1}{n_{r+s}}-\frac{b}{a},
 \end{equation}
 where $\frac{b}{a}=r+\frac{s}{2}-1$, and $n_i=g_i$ for $i\le r$, $n_i=2g_i$ for $i>r$.

\bigskip

It remains to discuss equation $(\ref{eq3})$ according to the possible values of the $g_i$'s. Since $g_i \ge 2$, we get from $(\ref{eq3bis})$ that $1 > \frac{r}{2} + \frac{s}{4}$, from which it follows immediately that $r \le 1$ and $2r+s \le 3$, so $\frac{b}{a}\in \{-\frac{1}{2},0, \frac{1}{2}\}$. Since $n_i$ divides $|G|$, $(\ref{eq3})$ forces $\frac{b}{a}> 0$ (and hence $\frac{b}{a}=\frac{1}{2}$),  unless $\frac{b}{a}=0$ and $r+s=1$. This last case can only occur if $r+\frac{s}{2}=1$, forcing $r=1,s=0$. We now examine the various possibilities.

\smallskip
\begin{itemize}
\item $\frac{b}{a}=0$ and $r=1$, $s=0$, then $g_1=|G|$ and $G$ is a \emph{cyclic group} of rotations around a single axis.
\end{itemize}

%\smallskip
%\begin{itemize}
%\item $\frac{b}{a}=0$ and $r=0$, $s=2$. This case is impossible, because the $n_i$'s are divisors of $|G|$.

% $$\frac{1}{|G|}=\frac{1}{n_1} + \frac{1}{n_2}$$

%Since $g_i=n_i/2$ divides $|G|$, this implies $|G|=g_1=g_2$ and $G$ has exactly two axes, both fixed by all of $G$. But any given isometry fixing two different axes must be a flip  around one of them. However the product of two orthogonal flips is a third orthogonal flip around a third axis. This contradicts the fact that there are only two axes. So this situation cannot occur.

%$(r,s)=(1,1)$, then $$\frac{1}{|G|} + \frac{1}{2}=\frac{1}{g_1} + \frac{1}{2g_2} > \frac{1}{2},$$ and thus either $g_1=2$ and $2g_2=|G|$, or $g_1=3$ and $g_2=2$ and $|G|=12$.

 %   In the first case, $G$ is a \emph{dihedral group} of orientation preserving isometries of a regular $n$-gon with $|G|=2n$ and $n$ is \emph{odd}. It has $\frac{|G|}{2}$ rotations around a single axis $x_0$ and $\frac{|G|}{2}$ flips around the lines joining the origin with each vertex of the polygon. We see that $n$ is odd if and only if all axis different from $x_0$ are congruent. So here $n$ is odd because we have exactly $2$ conjugacy classes of axes.

  %  In the second case, $G$ is the group of orientation preserving isometries of a regular \emph{tetrahedron}. We have $g_1=3$ rotations with angle a mutliple of $\frac{2\pi}{3}$ around each axis joining the origin with a vertex of the tetrahedron. The other rotations are flips of order $g_2=2$ around the line joining the origin with the center of each edge of the tetrahedron.
%\end{itemize}

\smallskip
\begin{itemize}
\item $\frac{b}{a}=\frac{1}{2}$, and $r=1$, $s=1$, $$\frac{1}{2}+\frac{1}{|G|}=\frac{1}{n_1} + \frac{1}{n_2}$$

We may assume $n_1 \le n_2$. This forces $n_1<4$ and hence $n_1=2,3$. There are thus two cases:

1) if $n_1=2$, then $n_2=2g_2=|G|$ and $G$ is a \emph{dihedral group} of order $2n$, with $n=g_2$ an \emph{odd integer}. $G$ is the group of orientation preserving isometries of a regular $n$-gon. Moreover in our case $n$ is \emph{odd} because there are only two conjugacy classes of axes.

2) if $n_1=3$, then one checks that $n_2=4$ and $|G|=12$. Here $G$ is the group of orientation preserving isometries of a \emph{regular tetrahedron}.
\end{itemize}

\smallskip
\begin{itemize}
\item $\frac{b}{a}=\frac{1}{2}$, $r=0$, $s=3$, $$\frac{1}{2}+\frac{1}{|G|}=\frac{1}{n_1} +\frac{1}{n_2}+ \frac{1}{n_3}$$

We may assume $n_1 \le n_2 \le n_3$. This forces $n_1<6$. Since $r=0$ all three $n_i$'s  are even $n_i=2g_i$, and thus either $n_1=2$ or $n_1=4$. The case $n_1=2$ is excluded as above. So $n_1=4$ and  thus $$\frac{1}{4}+\frac{1}{|G|}=\frac{1}{2g_2}+ \frac{1}{2g_3}.$$ We have the following cases:

1) $g_2=2$, then $|G|=n_3=2g_3$ and $G$ is a \emph{dihedral group} of order $2n$, with $n=g_3$ an \emph{even integer}. $G$ is the group of orientation preserving isometries of a regular $n$-gon. Moreover in our case $n$ is \emph{even} because there are exactly three conjugacy classes of axes. For example if $n=2$, $G \simeq (\Z/2\Z)^2$ and every non trivial element is a flip around one of three mutually orthogonal axes.

2) $g_2=3$, then $$\frac{1}{6}+ \frac{2}{|G|}=\frac{1}{g_3}.$$ This forces $g_3<6$, and hence three cases:

a) $g_3=3$, then $|G|=12$ and $G$ preserves a regular tetrahedron. This however is in contradiction with the assumption $s=3$, since there are only two conjugacy classes of axes in this case. So this case cannot occur.

b) $g_3=4$, then $|G|=24$ and $G \simeq S_4$ is the group of orientation preserving isometries of a \emph{cube} or regular \emph{octahedron}.

c) $g_3=5$, $G \simeq A_5$ is the group of orientation preserving isometries of a regular \emph{icosahedron} or \emph{dodecahedron}.

\end{itemize}

\section{Non-standard analysis and Jordan's unlimited numbers}\label{sec6}

As we have seen, Jordan gave no explicit bound on $J(n)$ in his article. Of course, this is not due to any fundamental ineffectiveness in the proof. Indeed, if one very carefully follows Jordan's argument, then it is possible to obtain this way a tower of exponentials type of bound, i.e. a tower $$10^{10^{\ldots ^{10}}}$$ of length $n$, see \cite{basset}. In fact Jordan himself seems to have been dissatisfied with his original exposition and devoted a second article \cite{jordan-napoli} where he rewrote his proof and explained why it is effective (even though he still did not supply a concrete bound). From a purely epistemological viewpoint, it is however interesting to consider how Jordan gets away with not writing down any bound whatsoever in his original memoir. In fact, in order to convince the reader of the soundness of his argument, he introduces a distinction between two kinds of numbers, that he calls limited and unlimited. Let us quote him \cite[p. 114]{jordan}:

\smallskip
``It is important for the study thereafter to make precise the meaning we attach to the words \emph{limited} and \emph{unlimited}. They are not synonymous to \emph{finite} and \emph{infinite}. We will say that a number is limited if it is smaller than a certain bound that has been determined. It follows from this definition that a finite number, about which we have no data, is unlimited; but it becomes limited as soon as we manage to assign a bound to it.''
\smallskip

This way, instead of saying that a certain quantity is bounded in terms of $n$ only, he says that the quantity is limited. While if it is not, it is unlimited and this is somehow leading to a contradiction when considering the class equation $(\ref{nn})$. A century and a half after Jordan, it is hard not to see there the premise of a way of thinking that prefigures non-standard analysis, where a new kind of numbers, the unlimited ones, is given an existence of its own.

In fact, it is possible to give a non-standard treatment of Jordan's proof, which we now sketch. Starting with a sequence of possible counter-examples to the theorem, one may take their ultraproduct, which becomes a certain infinite, pseudofinite, subgroup $\widehat{G}$ of $GL_n(\widehat{K})$. Here $\widehat{K}$ is the ultraproduct of algebraically closed fields $K_i$. A fan in $\widehat{G}$ is defined to be an internal diagonalizable subgroup $F$ such that each root $\alpha_{ij}: F \to \widehat{K}$, $f \mapsto \lambda_i(f)/\lambda_j(f)$ is either infinite or trivial. The induction hypothesis on the dimension allows one to assume that the centralizer $C(g)$ of every non-scalar element $g \in \widehat{G}$ admits a unique normal maximal fan $F_g$ whose index is finite.  Arguing as in Jordan's proof, we can partition the elements of $\widehat{G}$ into four parts: scalars $\Phi$, elements $g \notin \Phi$ with $F_g =\Phi$, elements $g \notin \Phi$ with $\Phi \lneq F_g$ and $F_g$ not maximal in $\widehat{G}$, elements $g \notin \Phi$ with $F_g$ maximal. Exploiting the fact that $\widehat{G}$ is pseudofinite, the second and third parts form a proportion of $r|\widehat{G}|$ of $\widehat{G}$, while the last part forms a proportion $(r'-\sum q_i^{-1})|\widehat{G}|$, where $r,r'$ are finite (standard) rationals and $q_i$ ($q_i=|N_{\widehat{G}}(F_{g_i})/\Phi|$) are a finite number of (non-standard) divisors of $|\widehat{G}/\Phi|$. Denoting by $g$ the non-standard integer $|\widehat{G}/\Phi|$, we thus have the equation:
$$\frac{1}{g}=r'' + \sum_i \frac{1}{q_i}$$
where $r''=1-(r+r')$. Taking the difference with standard parts we see that this implies:
$$\frac{1}{g}=\sum_i^{}{}^{'} \frac{1}{q_i}$$
for a subsum of the original sum. However the $q_i$'s are divisors of $g$ and the only way this can happen is if $q_i=g$ for some $i$. By $(\ref{borne})$ and $(\ref{normali})$, this means that $F_{g_i}$ is normal in $\widehat{G}$ and of finite index. This contradiction ends the proof.

\bigskip

The reader curious to take a look at Jordan's original article will see that the above non-standard treatment is in fact much closer to Jordan's own formulation of his proof than the exposition we have given of it in Section \ref{sec3}. Indeed Jordan does not talk about $M$-fans, but only defines fans. And he does so exactly as we did in the non-standard treatment above only using the word ``illimited'' in place of the word ``infinite''. Of course this definition can only make sense rigorously  if we place ourselves in a non-standard universe to begin with. So his proof is resolutely non-standard since its very inception.  His original formulation then reads \cite[\S 40 p. 114]{jordan}

\begin{theorem}(Jordan's theorem, original formulation) A finite subgroup $G$ of linear substitutions admits a unique maximal fan. It is normal and its index is a limited number.
\end{theorem}

\bigskip

To finish, we stress the key role of the finiteness of $G$ in Jordan's theorem. In Jordan's proof it is exploited arithmetically via the class equation. This is to be contrasted with Bieberbach's geometric argument via the commutator shrinking property, where finiteness is exploited via the element closest to the identity. 

\bigskip

We end this section by mentioning in passing some related recent developments around a question of Zilber \cite[Problem 6.3]{zilber} regarding pseudofinite groups. Recently, Nikolov, Schneider and Thom \cite{nikolov-schneider-thom} proved that  every homomorphism from a pseudofinite group to a compact Lie group has abelian-by-finite image, thus answering a conjecture of Pillay \cite[Conjecture 1.7]{pillay} and Zilber's question by the same token. 

Of course such a strong statement is more than enough to establish Jordan's theorem itself following the non-standard approach sketched above at least when the characteristic of $K$ is zero. Indeed, in this case $\widehat{K}$ can be taken to be isomorphic to $\C$ as any ultraproduct of countable algebraically closed fields of characteristic zero. Furthermore, we may assume that $\widehat{G}$ lies in the internal set of unitary matrices $\U_n(\widehat{K})$, in other words that $\widehat{G}$ is a subgroup of a compact Lie group. By Nikolov-Schneider-Thom, this implies that  $\widehat{G}$  is abelian-by-finite, which is the desired contradiction. 

Of course the theorem of Nikolov, Schneider and Thom lies much deeper than Jordan's theorem, because it applies to any pseudofinite group and not only those lying in some $\GL_n$ for some fixed $n$. Their proof relies on the deep results of Nikolov and Segal \cite{nikolov-segal} about commutator width in finite groups.

\section{Bounds on $J(n)$} \label{sec7} 

To conclude, we briefly survey the history around Jordan's theorem and how bounds on $J(n)$ have sharpened over time: 

\bigskip

Jordan (1878): no bound (in fact: tower of exponentials \cite{basset}).

Blichfeldt (1905): $\exp(O(n^3))$.

Bieberbach (1911): $(1+32^4n^{10})^{2n^2}$.

Frobenius (1911): $(\sqrt{8n}+1)^{2n^2}$.

Blichfeldt (1917):  $n!6^{(n-1)(\pi(n+1)+1)}$ [$\pi(x)$ is the number of prime numbers $\leq x$].

Weisfeiler (1984): $(n+1)! e^{O(\log n)^2}$ (using CFSG).

Collins (2007): $(n+1)!$ for $n\ge 71$ (using CFSG).

\bigskip

In \cite[I p. 396]{blichfeldt} Blichfeldt, who had just completed his dissertation under Sophus Lie, shows that no prime  $p\ge n(2n-1)$ divides the order of a finite primitive subgroup of $\GL_n(\C)$ and in \cite[II p. 321]{blichfeldt} he obtains bounds for the $p$-exponent of the order of a primitive subgroup, in particular this exponent is at most $1$ if $p>n$ (see \cite[\S 67]{blichfeldt-FCL}). These yield a bound of the order of $\exp(O(n^3))$ for arbitrary finite subgroups (implicit in \cite[III p. 232]{blichfeldt}).

In his 1917 monograph \cite{blichfeldt-FCL} Blichfeldt furthers his earlier results incorporating a geometric argument inspired by  Bieberbach's argument \cite{bieberbach} and Frobenius' improvement \cite{frobenius}. He shows in \cite[\S 73]{blichfeldt-FCL} that an abelian subgroup of a primitive group must have order at most $6^{n-1}$ times the size of the group of scalar matrices. This is based on a lemma \cite[\S 70]{blichfeldt-FCL} according to which in a finite primitive subgroup, any transformation whose eigenvalues are concentrated in an arc of length at most $\frac{2\pi}{3}$ centered at one of them on the unit circle, must be scalar. This lemma is closely related to the Bieberbach-Frobenius proofs. And finally he derives the bound $n!6^{(n-1)(\pi(n+1)+1)}$ on $J(n)$, where $\pi(n)$ is the number of primes $\leq n$. See \cite[ch. 30]{dornhoff} for a thorough treatment of Blichfeldt's bound and \cite{robinson,robinson2} for recent improvements on Blichfeldt's lemma. In fact Blichfeldt claimed that $6$ could be replaced by $5$, but no proof of this has appeared. The three-author book \cite{miller-blichfeldt-dickson}, which is dedicated to Camille Jordan, also contains a summary of these results.

% shows ($\S 109$) that a finite subgroup of $\GL_n(\C)$ contains an abelian Hall $\pi$-subgroup, where $\pi$ is the family of all primes $>n+1$.

Other excellent expositions of Blichfeldt's bound are contained in \cite{speiser} and \cite{isaacs}. See also \cite[ch. 5]{dixon} for a treatment of Blichfeldt's earlier results and a proof of Jordan's theorem using the Bieberbach-Frobenius argument.

Brauer \cite{brauer} conjectured that Blichfeldt's bound could be improved to one of the form $e^{O(n\log n)}$ and indeed he was able to achieve it under certain hypotheses. In fact, for finite solvable subgroups Dornhoff proved an exponential bound $2^{4n/3}3^{10n/9-1/3}$  that is even sharp for infinitely many $n$'s \cite[Th. 36.4]{dornhoff}.

Nevertheless, Blichfeldt's second bound of the form $e^{O(n^2/\log n)}$ seems to be the best one available without the classification of finite simple groups (CFSG). This small looking gain of a factor $(\log n)^2$ in the exponent  compared to the Bieberbach-Frobenius bound can sometimes prove important, as we have found out in \cite{breuillard-pisier}.

Shortly before disappearing while hiking\footnote{See \url{https://en.wikipedia.org/wiki/Boris_Weisfeiler}} in Chile, B. Weisfeiler announced a bound on $J(n)$ of $e^{O(n\log n)}$ quality \cite{weisfeiler}. His unpublished manuscript has now been typed-up and is available online \cite{weisfeiler-online}. Finally, more recently, M. Collins \cite{collins, collins2} has improved the bound  for $n$ large to one that is sharp, namely $(n+1)!$, thus closing a long chapter in the history of finite linear groups.

\setcounter{tocdepth}{1}

\end{document}